\documentclass{amsart}
\usepackage{amssymb}
\usepackage{amsmath}
\numberwithin{equation}{section}
\usepackage{amsthm}
\usepackage{appendix}
\usepackage{bookmark}
\usepackage{setspace}
\usepackage{mathrsfs}
\usepackage{mathtools}
\usepackage{tikz-cd}
\usepackage{manfnt}
\usepackage{fullpage}
\newtheorem{thm}{Theorem}
\numberwithin{thm}{section}
\newtheorem{prop}[thm]{Proposition}
\newtheorem{lem}[thm]{Lemma}
\newtheorem{cor}[thm]{Corollary}
\theoremstyle{remark}
\newtheorem{rem}[thm]{Remark}
\theoremstyle{definition}

\newtheorem{eg}[thm]{Example}

\newcommand{\1}{\mathbf{1}}

\newcommand{\C}{\mathbf{C}}
\newcommand{\cali}[1]{\mathcal{#1}}
\newcommand{\chx}[1]{\langle #1\rangle}

\newcommand{\comment}[1]{}

\newcommand{\F}{\mathbf{F}}

\newcommand{\got}[1]{\mathfrak{#1}}

\newcommand{\plim}{\varprojlim}

\newcommand{\Q}{\mathbf{Q}}

\newcommand{\R}{\mathbf{R}}

\newcommand{\Scr}[1]{\mathscr{#1}}

\newcommand{\Z}{\mathbf{Z}}

\newcommand{\Zp}{\mathbf{Z}_p}

\newcommand{\mes}{\mathrm{Mes}}

\title{Sum Expressions for Kubota-Leopoldt $p$-adic $L$-functions}

\author{Luochen Zhao}

\date{Jan 18, 2022}



\address{Department of Mathematics, Johns Hopkins University, 404 Krieger Hall, 3400 N.~Charles Street, Baltimore, MD 21218, USA}
\email{lzhao39@jhu.edu}
\begin{document}
	
	\maketitle
	
	\begin{abstract}
		When $p$ is an odd prime, Delbourgo observed that any Kubota-Leopoldt $p$-adic $L$-function, when multiplied by an auxiliary Euler factor, can be written as an infinite sum. We shall establish such expressions without restriction on $p$, and without the Euler factor when the character is nontrivial, by computing the periods of appropriate measures. As an application, we will reprove the Ferrero-Greenberg formula for the derivative $L_p'(0,\chi)$. We will also discuss the convergence of sum expressions in terms of elementary $p$-adic analysis, as well as their relation to Stickelberger elements; such discussions in turn give alternative proofs of the validity of sum expressions.
	\end{abstract}

	\setcounter{tocdepth}{1}
	
	\section{Introduction}
	
	\subsection{Motivation}
	\label{subsec.introduction-motivation}
	 
	The facts that we mention below without proof, if not specified otherwise, can be found in Chapter 3 of \cite{Hida}, Chapter 2 of \cite{Kobl}, Chapters 4 and 10 of \cite{Lang}, and Chapters 5 and 12 of \cite{Wash}. All limits, when not emphasized, are with respect to the $p$-adic topology.
	
	\vspace{3mm}
	
	Given an even Dirichlet character $\chi$, it is well known that there exists a unique $p$-adic meromorphic function $L_p(s,\chi)$ for $s\in \Z_p$ that satisfies the interpolation property:
	\begin{align}
		\label{equation.interpolation-classical}
		\text{for all integer }n\ge 1,\ L_p(1-n,\chi) = (1-\chi\omega^{-n}(p)p^{n-1})L(1-n,\chi\omega^{-n}),
	\end{align}
	
	where $L(s,\chi)$ for $s\in \C$ is the usual Dirichlet $L$-function. As is customary, we refer to these $L_p(s,\chi)$ as Kubota-Leopoldt $p$-adic $L$-functions, or just $p$-adic $L$-functions for short. Also, when $\chi$ is the trivial character $\1$ (resp.~is of conductor $p^t$ for some $t\in\Z_{>0}$), we refer to $L_p(s,\1)$ (resp.~$L_p(s,\chi)$) as the $p$-adic zeta function (resp.~a twist of the $p$-adic zeta function). 
	
	\vspace{3mm}
	
	One construction of the Kubota-Leopoldt $p$-adic $L$-functions, as initiated by Mazur, is from their integral representations. For example, when the conductor of $\chi$ is a power of $p$, for any $c\in \Z_{>0}$ prime to $p$, there exists a $p$-adic measure $\mu_{1,c^{-1}}$ \cite[§II.5]{Kobl} (denoted $E_{1,c}$ in \cite[§2.2]{Lang}), such that \cite[p107]{Lang}
	\begin{align}
		\label{equation.integral representation}
		-L_p(s,\chi) = \frac{1}{1-\chi(c)\chx{c}^{1-s}}\int_{\Z_p^\times}\omega^{-1}\chi(x) \chx{x}^{-s} \mu_{1,c^{-1}}(x).
	\end{align}
	
	In a different flavor, thanks to works of Delbourgo, it is also possible to write $p$-adic $L$-functions more explicitly as a conditionally convergent infinite sum. For instance, when $p$ is odd, Delbourgo \cite{De06} obtained the following sum expression:
	\begin{align}
		\label{equation.delbourgo-2}
		L_p(s,\omega^{\beta+1}) = \frac{1}{2(1-\omega^{\beta+1}(2)\chx{2}^{1-s})} \sum_{n\ge 1}
		\left( \sum_{\substack{p^{n-1}\le m<p^n \\ p\nmid m}}
		\frac{\omega^\beta(m)}{\chx{m}^s}(-1)^{m+1}
		\right),
	\end{align}
	
	where $\beta\in \Z/(p-1)$ and $s\in \Z_p$. His method, which crucially exploits a periodicity nature of the measure $\mu_{1,2^{-1}}$, will be explained in §\ref{subsec.introduction-method} below.
	
	\vspace{3mm}
	
	We conclude by summarizing the importance of sum expressions beyond their inherent appeal: First, they provide a straightforward way to numerically approximate the special values of $L_p(s,\chi)$ for $s\in\Z_{>1}$, which are outside of the classical range. In turn, the nonvanishing of such values would imply the finiteness of certain Iwasawa modules (see, e.g., \cite[Proposition 3.3.7]{Co15}). Second, their theoretical significance is illustrated in the study of derivatives, and higher derivatives, of $p$-adic $L$-functions, thanks to their explicit nature. We shall demonstrate this point by giving a direct proof of the Ferrero-Greenberg formula for the derivative of the $p$-adic $L$-function in §4, a result that is known to be a key link in the proof of Gross-Stark conjecture in the case of $\Q$ \cite[§4]{Gr81}.
	
	\subsection{Notation}
	\label{subsec.introduction-notation}
	
	Fix a rational prime $p$. Let $\C_p$ be the completion of an algebraic closure $\bar{\Q}_p$ of $\Q_p$, and let $\got{o}_p$ be its ring of integers. Fix also an algebraic closure $\bar{\Q}$ of $\Q$ and embeddings of $\bar{\Q}$ into $\C$ and $\C_p$, so we can identify $\bar{\Q}$ as a subfield of both. For $a,b\in \C_p$ and $r\in \Z_{\ge 0}$, by $a\equiv b \bmod p^r$ we mean that $a-b\in p^r\got{o}_p$. Thus if $a,b\in \Q_p$, then $a\equiv b\bmod p^r$ means $a-b\in p^r\Z_p$. In the rest of this article, unless specified otherwise, $N$ will always denote a positive integer bigger than 1 and is prime to $p$, and $q=p^f>1$ will denote the smallest power of $p$ such that $q\equiv 1\bmod N$. The letter $\chi$ shall denote a Dirichlet character of conductor $N$, while $\psi$ shall denote one of conductor a power of $p$, which will be simply denoted by $\mathbf{1}$ when it is trivial. When $h>1$ is an integer, for any $a\in \Z/h$, we denote by $a^\flat_h$,$a^\sharp_h$ the unique integers such that $a^\flat_h\in [0,h)$, $a^\sharp_h\in (0,h]$ and $a\equiv a^\flat_h\equiv a^\sharp_h\bmod h$.
	
	\subsection{Overview of the main result}
	\label{subsec.introduction-overview}
	
	Our main achievement in this paper is the establishment of sum expressions for all Kubota-Leopoldt $p$-adic $L$-functions in complete generality, extending the previous works of Delbourgo \cite{De06,De09,De09-2}. The exact statement is as follows:
	\begin{thm}
		\label{thm.main}
		Let $p$ be a prime, $N\in\Z_{>1}$ be prime to $p$, and let $q=p^f>1$ be the smallest $p$-power such that $p^f\equiv 1\bmod N$. Suppose $\chi$ is a Dirichlet character of conductor $N$, and $\psi$ is one of conductor a power of $p$. Then we have
		\begin{enumerate}
			\item[(i)] for the (twist of) $p$-adic zeta function:
			\begin{align*}
				L_p(s,\psi\omega) = -\frac{1}{1-\psi\omega(N)\chx{N}^{1-s}}\lim_{n\to \infty}\sum_{1\le m< q^n, p\nmid m} \frac{\psi(m)}{\chx{m}^s}m^\sharp_N.
			\end{align*}
			
			\item[(ii)] for the (twist of) $p$-adic Dirichlet $L$-function:
			\begin{align*}
				L_p(s,\chi\psi\omega) = -\lim_{n\to \infty} \sum_{1\le m< q^n, p\nmid m} \frac{\psi(m)}{\chx{m}^s}\sum_{1\le a< m^\flat_N} \chi(a).
			\end{align*}
		\end{enumerate}
	\end{thm}
	
	\begin{rem}
		After the draft of this paper has been finalized, we are informed that Knospe and Washington \cite[Theorem 2.4, Corollary 2.5]{KW21} have obtained (i) and a weaker version of (ii). For this reason we only sketch the proof of (i) in §\ref{subsec.introduction-method}. Our formula in (ii) may be regarded as stronger as it removes the auxiliary Euler factor away from $p$ that is present in \cite{KW21}.
	\end{rem}
	
	\begin{rem}
		While the Kubota-Leopoldt $p$-adic $L$-functions are in fact defined on the larger domain \cite[Theorem 5.9]{Wash}
		\begin{align*}
			\{s\in \C_p: |s|_p<|2p^{\frac{p-2}{p-1}}|_p^{-1}\},
		\end{align*}
		
		in this article we will exclusively treat the case when $s$ varies in $\Z_p$, for it is sufficient for our purposes, and for saving us from additional discussions on analysis. Basically, when $s\notin \Z_p$, using our approach one can derive the same limit formulas, and even estimates on the error term if one so desires, by using the asymptotic identity $\chx{x}^s= \chx{a}^s + O(p^n s)$ for all $a\in \Z_p^\times$, $x\in a+p^n\Z_p$ and $s$ in the above range.
	\end{rem}

	\begin{rem}
		It is worth pointing out that, when $s=1$, $\psi=\1$ and $N=2$, the formula in (i) is virtually due to Koblitz \cite[p461, Remark 1]{Ko79} (see also \cite[Theorem 59.1]{Schi}). Formulas of a similar flavor are recorded in Remark \ref{remark.leopoldt}.
	\end{rem}
	
	\subsection{Novelty in this paper} 
	\label{subsec.introduction-whatsnew}
	
	As remarked earlier, the statement about the sum expressions for $p$-adic Dirichlet $L$-functions without the Euler factor is new. The explicit period formulas that lead to them, however, were implicit in the works of Iwasawa on Stickelberger elements, and were studied in the works of Ferrero \cite{Fe78} and Ferrero-Washington \cite{FW79}. Nevertheless, the measure-theoretic form in which the formulas is stated does not seem to have been recorded in the literature, not does our proof of them by computing certain periods of a particular rational function. Finally, the analytic formulas listed in Remark \ref{remark.leopoldt} (when $N\ne 2$) and the expressions of higher derivatives of $p$-adic $L$-functions in \eqref{equation.taylor} seem to be new.
	
	\vspace{3mm}
	
	Another aspect we supplement to previous works on sum expressions is an alternative explanation of their convergence. This only uses some elementary $p$-adic analysis, as well as a trick of Ferrero-Greenberg. We believe it makes these infinite sums less mysterious, and exemplifies the naturalness and significance of the Ferrero-Greenberg permutation in their studies. The last point is further strengthened in our re-derivation of the derivative formula, a useful application that is not pointed out in preceding works.

	\subsection{The method}
	\label{subsec.introduction-method}
	
	To establish Theorem \ref{thm.main}, we will follow the method of Delbourgo \cite{De06}, complemented by its measure-theoretic reflection. In concrete terms, we invoke a machinery that yields infinite sum expressions of $p$-adic $L$-functions based on three ingredients:
	\begin{enumerate}
		\item [(a)] the integral representation of the $p$-adic $L$-functions;
		
		\item [(b)] computability of various periods $\mu(a+p^n\Z_p)$ for all $a\in \Z_p$ and $n\in \Z_{\ge 0}$, where $\mu$ is the attached measure in (a);
		
		\item [(c)] \textit{uniform periodicity} of the above periods, i.e., there exists $f\in \Z_{>0}$ such that for all $m\in \Z_{\ge 0}$ and all $n,n'\in \Z_{\ge 0}$ such that $p^n> m$, $p^{n'}> m$ and $n\equiv n' \bmod f$, we have $\mu(m+p^n\Z_p) = \mu(m+p^{n'}\Z_p)$.
	\end{enumerate}
	
	We explain here how these can be assembled together to give a quick proof of the first part of Theorem \ref{thm.main}: The integral representation is provided by \eqref{equation.integral representation} with $c=N$, and we have the explicit formula \cite[p38]{Lang}
	\begin{align*}
		\mu_{1,N^{-1}}(a+p^n\Z_p) = \frac{a}{p^n}-\frac{1}{2} - N\left(\frac{(a/N)^\flat_{p^n}}{p^n}-\frac{1}{2}\right),
	\end{align*}
	
	where $n\in \Z_{\ge 0}$ and $0\le a<p^n$. It can be worked out that $\mu_{1,N^{-1}}(a+p^n\Z_p) = (\frac{m}{p^n})^\sharp_N - \frac{N+1}{2}$ (\textit{cf.}~\cite[Theorem 3.1]{KW21}). Therefore, by taking the Riemann sums of \eqref{equation.integral representation}, we have
	\begin{align*}
		-(1-\psi\omega(N)\chx{N}^{1-s})L_p(s,\psi\omega) = \lim_{n\to\infty}\sum_{1\le m< p^n, p\nmid m} \frac{\psi(m)}{\chx{m}^s}\left[\left(\frac{m}{p^n}\right)^\sharp_N -\frac{N+1}{2}\right].
	\end{align*}
	
	This is not yet an infinite sum, since the $m$-th coefficient, $a_m(n) = (\frac{m}{p^n})^\sharp_N-\frac{N+1}{2}$, depends on $n$. Still, one can see that their dependence on $n$ is uniform and periodic in the sense of (c) above. Thus, taking only $n$ to be divisible by $f$, we have the formula in Theorem \ref{thm.main}.(i), by noting that $\lim_{n\to \infty} \sum_{1\le m< p^n, p\nmid m} \psi(m)\chx{m}^{-s}=0$.

	\begin{rem}
		When $N=2$ this reproduces the  argument in \cite{De06}, by \cite[Proposition 4.3.4]{Lang}.
	\end{rem}
	
	\subsection{Further developments}
	\label{subsec.introduction-further}
	
	The sum expressions also exist for $p$-adic Hecke $L$-functions of totally real fields under an analog of Heegner hypothesis introduced by Cassou-Nogu\`es, and will be treated in a forthcoming paper. These sum expressions thus give rise to generalizations of the Ferrero-Greenberg derivative formula, and we hope they could shed more light on the Gross-Stark conjecture. 
	
	\vspace{3mm}
	
	In a separate paper, we will discuss the possibility of writing the $p$-adic $L$-functions of cuspidal forms as an infinite sum. In this case, the integral representation (of an allowable root of the Hecke polynomial) is known by works of Amice-V\'elu and Vishik (see, e.g.,\cite{MTT}), and the periods are in a sense computable, being modular integrals. We are, however, not certain if uniform periodicity holds for them. For example, when the weight is $2$ and the form is rational, the periodicity implies the attached elliptic curve has multiplicative reduction at $p$, as well as the vanishing of abundant modular integrals.

	\subsection{Acknowledgement}
	
	The author is very thankful to Professor Antonio Lei for suggesting this direction and the various topics it covers, as well as reviewing, discussing and correcting the many drafts that led to this paper; we owe this paper to him. The author is also thankful to Professor Daniel Delbourgo for his advice and for pointing out the paper \cite{KW21}. Finally we are grateful to the referee for a careful reading and helpful suggestions.

	\section{Background on $p$-adic Measures}
	\label{sec.background}
	
	In this section we give a quick recapitulation of $p$-adic measure theory; the details can be found in the standard textbooks listed in §\ref{subsec.introduction-motivation}. 
	
	\vspace{3mm}
	
	Let $R$ be the ring of integers of some $p$-adic field over $\Q_p$. Denote by $\cali{C}(\Z_p,R)$ the set of all continuous functions from $\Z_p$ to $R$, topologized by the sup-norm. Recall a measure on $\Z_p$ valued in $R$ is a continuous $R$-linear map $\cali{C}(\Z_p,R) \to R$, and we denote by $\mes(\Z_p,R)$ the set of all such linear maps. For $f\in \cali{C}(\Z_p,R)$ and $\mu\in \mes(\Z_p,R)$, their pairing will be written as $\int_{\Z_p} f(x)\mu(x)$, which is explicitly given by the limit of Riemann sums, $\lim_{n\to \infty}\sum_{0\le a< p^n}f(a)\mu(a+p^n\Z_p)$.
	
	\vspace{3mm}
	The space of measures can in fact be understood algebraically. By restricting the Amice transform, defined a priori on locally analytic distributions, we have an isomorphism $\cali{F}: \mes(\Z_p,R) \xrightarrow{\sim} R[[t-1]]$, where $R[[t-1]]$ is the ring of formal power series in $t-1$, equipped with the $(p,t-1)$-adic topology. As such, for any $F\in R[[t-1]]$ and $\mu\in \mes(\Z_p,R)$, we shall denote $\cali{F}(\mu)(t)$ by $\Scr{A}_\mu(t)$, and $\cali{F}^{-1}(F(t))$ by $\mu_F$. The power series realization of a measure, to us, is most useful in the following
	\begin{prop}
		For any $\mu\in \mes(\Z_p,R)$, $n\in\Z_{\ge 0}$ and $a\in \Z_p$, we have
		\begin{align*}
			\mu(a+p^n\Z_p) = \frac{1}{p^n}\sum_{\zeta:\zeta^{p^n}=1}\zeta^{-a}\Scr{A}_\mu(\zeta),
		\end{align*}
		
		where the sum is over all $p^n$-th root of unity in $\C_p$.
	\end{prop}
	
	Symmetrically, for $F(t)\in R[[t-1]]$ and all $n\in\Z_{\ge 0}$ and $a\in \Z_p$, we define 
	\begin{align}
		\label{equation.period}
		\Omega_F(a,n) = \frac{1}{p^n}\sum_{\zeta:\zeta^{p^n}=1}\zeta^{-a}F(\zeta),
	\end{align}
	
	which, through the Amice transform, equates $\mu_F(a+p^n\Z_p)$. We will often refer to $\Omega_F(a,n)$'s as \textit{periods} (attached to $F$ or $\mu_F$). In turn, we have
	\begin{align}
		\label{equation.riemann-power series}
		\int_{\Z_p} f(x)\mu_F(x) = \lim_{n\to \infty}\sum_{0\le a< p^n}f(a)\Omega_F(a,n).
	\end{align}
	
	We finish this section by introducing our protaganist, the rational function
	\begin{align}
		\label{equation.Lchi}
		L_\chi(t) = \sum_{1\le a< N} \chi(a)\frac{t^a}{t^N-1}.
	\end{align}
	
	Here, as set in §\ref{subsec.introduction-notation}, $\chi$ denotes a Dirichlet character of conductor $N>1$ that is prime to $p$. Denote further by $\Z_p[\chi]$ the ring obtained from $\Z_p$ by adjoining values of $\chi$. Since $N$ is prime to $p$, we see that $L_\chi(t)$, as an element in the fraction field of $\Z_p[\chi][[t-1]]$, is regular except potentially having a pole of order $1$ at $t=1$. The singularity is, however, superfluous, since
	\begin{align}
		\label{equation.power series Lchi}
		\begin{split}
			\sum_{1\le a<N}\chi(a)\frac{t^a}{t^N-1} 
			&= \sum_{1\le a<N}\chi(a)\frac{t^a-1}{t^N-1}+\frac{1}{t^N-1}\sum_{1\le a<N}\chi(a) \\
			&= \sum_{1\le a<N}\chi(a)\frac{t^a-1}{t^N-1}\\
			&= \sum_{1\le a<N}\chi(a)\frac{a+\frac{a(a-1)}{2}(t-1)+\cdots}{N + \frac{N(N-1)}{2}(t-1) + \cdots}.
		\end{split}
	\end{align}
	
	Thus $L_\chi(t)$ belongs to $\Z_p[\chi][[t-1]]$, and $L_\chi(1)=\sum_{0\le a<N}\chi(a)\frac{a}{N}$, which is further equal to $-L(0,\chi)$ \cite[Corollary 2.3.2]{Hida}. Finally we remind the reader of the role $L_\chi$ plays in the integral representation: let $\psi$ be any Dirichlet character of a $p$-power conductor, then
	\begin{align}
		\label{equation.integral-representation-dirichlet}
		-L_p(s,\chi\psi) = \int_{\Z_p^\times}\omega^{-1}\psi(x)\chx{x}^{-s} \mu_{L_\chi}(x).
	\end{align}
	
	As \eqref{equation.integral-representation-dirichlet} is not found in \cite{Hida}, we give it a very brief account. By a density argument, it suffices to prove the above integral of $\mu_{L_\chi}$ interpolates $L$-values for all $s=-k\in \Z_{\le 0}$. Also, we may work with $\psi\omega^{k+1}$ instead of $\psi$, where $\psi$ is of conductor $p^t$. Then we have (\textit{cf.}~\cite[p485]{Ka75}):
	\begin{align*}
		-\sum_{n\ge 1, p\nmid n} \chi\psi(n)n^k = -\sum_{n\ge 1, p\nmid n} \chi\psi(n)n^k t^n\Big|_{t=1} = \left(t\frac{d}{dt}\right)^k\Big\{\sum_{\substack{1\le a<Np^{\max\{t,1\}}\\ p\nmid a}}\frac{\chi\psi(a)t^a}{t^{Np^{\max\{t,1\}}}-1}\Big\}\Big|_{t=1},
	\end{align*}
	
	where the leftmost quantity is no other than $-(1-\chi\psi(p)p^k)L(-k,\chi\psi)$, and the rightmost quantity, by the Amice transform, is $\int_{\Z_p} x^k\mu_{[\psi|_{\Z_p^\times}]L_\chi}(x) = \int_{\Z_p^\times}x^k\psi(x)\mu_\chi(x)$ (see \cite[§3.5]{Hida}).
	
	\begin{rem}
		We will not need it, but the same argument establishes the identity
		\begin{align*}
			-L(-k,\chi) = \int_{\Z_p} x^k \mu_{L_\chi}(x).
		\end{align*}
		
		When $k=0$, this has been observed in the form of $L_\chi(1) = -L(0,\chi)$. 
	\end{rem}
	
	In what follows we shall often write $\mu_{L_\chi}$ simply as $\mu_\chi$.
	
	\section{Explicit Period Formulas and Sum Expressions}
	\label{sec.sumexpr}
	
	Invoking the machinery introduced in §\ref{subsec.introduction-method}, in this section we will establish the second part of Theorem \ref{thm.main} for any Dirichlet character $\chi$ of conductor $N>1$, where $\gcd(N,p)=1$. Note that ingredient (a) is provided by \eqref{equation.integral-representation-dirichlet}. As such, below we will start by establishing the explicit formulas of the periods $\Omega_{L_\chi}(m,n)=\mu_\chi(m+p^n\Z_p)$ attached to the power series $L_\chi$, or $\mu_\chi$, for all $n\in \Z_{\ge 0}$ and $m\in \Z_p$; this fulfils the requirement (b). We will see from the explicit formula \eqref{equation.period-formula} that (c) automatically follows.
	
	\vspace{3mm}
	
	Recall that $L_\chi(t) = \sum_{1\le a<N} \chi(a)\frac{t^a}{t^N-1}$ is an element in $\Z_p[\chi][[t-1]]$. In §\ref{sec.background} we have seen that
	\begin{align}
		\label{equation.Lchi(1)}
		L_\chi(1) = \sum_{1\le a<N}\chi(a)(a/N) = -L(0,\chi).
	\end{align}
	
	For $\zeta\ne 1$, we have the following
	\begin{lem}
		The following identity holds in $\bar{\Q}$, and thus in $\C_p$:
		\begin{align}
			\label{equation.Lchi-non1}
			\sum_{\zeta\ne 1:\zeta^{p^n}=1}L_\chi(\zeta)\zeta^{-m} = \sum_{1\le a< N}\chi(a)(\frac{m-a}{N})^\flat_{p^n}.
		\end{align}
	\end{lem}
	
	\begin{proof}
		It suffices to prove this in $\C$. For this we introduce an auxillary parameter $x\in \R$ such that $0<x<1$, and observe the following archimedean limit:
		\begin{align*}
			\sum_{\zeta\ne 1:\zeta^{p^n}=1}L_\chi(\zeta)\zeta^{-m} = \lim_{x\to 1^-}\sum_{\zeta\ne 1:\zeta^{p^n}=1}\zeta^{-m}\sum_{1\le a<N}\chi(a)\frac{\zeta^a}{x\zeta^N-1}.
		\end{align*}
		
		Before taking the limit, the right hand side can be expanded as:
		\begin{align*}
			\sum_{\zeta\ne 1:\zeta^{p^n}=1}\sum_{1\le a<N}\chi(a)\frac{\zeta^a}{x\zeta^N-1}\zeta^{-m}
			&=\sum_{\zeta:\zeta^{p^n}=1}\sum_{1\le a<N}\chi(a)\frac{\zeta^a}{x\zeta^N-1}\zeta^{-m} - \frac{1}{x-1}\sum_{1\le a<N}\chi(a)\\
			&=-\sum_{\zeta:\zeta^{p^n}=1}\sum_{1\le a<N}\chi(a)\zeta^{a-m}\frac{1}{1-x\zeta^N}\\
			&=-\sum_{1\le a<N}\chi(a)\sum_{\zeta:\zeta^{p^n}=1}\zeta^{a-m}\sum_{k\ge 0}x^k\zeta^{Nk}\\
			&=-\sum_{1\le a<N}\chi(a)\sum_{k\ge 0}x^k\sum_{\zeta:\zeta^{p^n}=1}\zeta^{a-m+Nk}\\
			&=-p^n \sum_{1\le a<N}\chi(a)\sum_{\substack{k\ge 0\\ Nk\equiv m-a\bmod p^n}}
			x^k.
		\end{align*}
		
		As we have assumed that $\gcd(p,N)=1$, the above can be malleated further to
		\begin{align*}
			-p^n \sum_{1\le a<N}\chi(a)\sum_{\substack{k\ge 0\\ k\equiv \frac{m-a}{N}\bmod p^n}}
			x^k
			&=
			-p^n\sum_{1\le a<N}\chi(a)\frac{x^{(\frac{m-a}{N})^{\flat_{p^n}}}}{1-x^{p^n}}\\
			&=p^n\sum_{1\le a<N}\chi(a)\frac{1-x^{(\frac{m-a}{N})^\flat_{p^n}}}{1-x^{p^n}}\\
			&\to \sum_{1\le a<N}\chi(a)(\frac{m-a}{N})^\flat_{p^n} \quad(x\to 1^-).
		\end{align*}
	\end{proof}

	We can further analyze $(\frac{m-a}{N})^\flat_{p^n}$. Clearly, there is a unique $h(a,m) = h_n(a,m) \in \Z$ such that $(\frac{m-a}{N})^\flat = \frac{m-a+h(a,m)p^n}{N}$. So, we may re-structure the above sum as:
	\begin{align*}
		\sum_{1\le a<N}\chi(a)(\frac{m-a}{N})^\flat
		&= \sum_{1\le a<N}\chi(a)\frac{m-a+p^n h(a,m)}{N}\\
		&= \frac{m}{N}\sum_{1\le a<N}\chi(a) -\frac{1}{N}\sum_{1\le a<N}\chi(a)a +\frac{p^n}{N}\sum_{1\le a<N}\chi(a)h(a,m)\\
		&=L(0,\chi) + \frac{p^n}{N}\sum_{1\le a<N}\chi(a)h(a,m).
	\end{align*}
	
	Combinging \eqref{equation.Lchi(1)} and \eqref{equation.Lchi-non1}, we conclude that
	\begin{align}
		\label{equation.period-implicit}
		\Omega_\chi(m,n) = \frac{1}{p^n}\sum_{\zeta:\zeta^{p^n}=1}\zeta^mL_\chi(\zeta^{-1}) = \frac{1}{N}\sum_{1\le a<N}\chi(a)h(a,m).
	\end{align}
	
	\vspace{3mm}
	
	We now compute $h(a,m) = h_n(a,m)$. Recall that for $b\in \Z/N$, we use $b^\flat_N$ to denote the unique integer in $[0,N)$ such that $b^\flat_N \equiv b\bmod N$. 
	\begin{prop}
		\label{prop.h(a,m)}
		For $n\ge 0$, $0\le a<N$ and $0\le m<p^n$, we have $0\le h_n(a,m) <N$, whence $h_n(a,m) = (\frac{a-m}{p^n})^\flat_N$. In turn, for $n,a,m$ in the said range, the function $h_n(a,m)$ is periodic in $n$ with period $f$, the smallest positive integer such that $p^f \equiv 1\bmod N$, and is periodic in $m$ with period $N$. As a special case, if $p^n\equiv 1\bmod N$, then $h_n(a,m) = (a-m)^\flat_N$.
	\end{prop}
	
	\begin{proof}
		Since $(\frac{m-a}{N})^\flat = \frac{m-a+h(a,m)p^n}{N} <p^n$, $a< N$ and $m\ge0$, we have:
		\[\frac{-N+h(a,m)p^n}{N}< \frac{m-a+h(a,m)p^n}{N} \le p^n-1,\]
		
		which gives $h(a,m)<N$. On the other hand, because $(\frac{m-a}{N})^\flat \ge 0$, $a\ge 0$ and $m<p^n$, we have 
		\[\frac{p^n + h(a,m)p^n}{N}> \frac{m-a+h(a,m)p^n}{N}\ge 0.\]
		
		Therefore $h(a,m) > -1$, which implies $h(a,m)\ge 0$. Now, we can pin down the value of $h_n(a,m)$: By definition, $\frac{m-a+p^n h(a,m)}{N}$ is an integer, so $m-a+h(a,m)p^n \equiv 0 \bmod N$, from which we have $h(a,m) \equiv \frac{a-m}{p^n} \bmod N$. The inequality $0\le h(a,m) <N$ then forces $h(a,m)$ to be $(\frac{a-m}{p^n})^\flat_N$. The remaining statements are clear.
	\end{proof}
	
	With the value of $h(a,m)$ uncovered, we can now establish the explicit period formula:
	\begin{thm}
		\label{thm.explicit-period-formula}
		Let $\chi$ be a Dirichlet character of conductor $N>1$ that is prime to $p$, and let $\mu_\chi = \mu_{L_\chi}$ be the $p$-adic measure attached to $L_\chi$ \eqref{equation.Lchi}. Then for any $0\le m<p^n$, we have
		\begin{align}
			\label{equation.period-formula}
			\mu_\chi(m+p^n\Z_p) = \frac{1}{N}\sum_{1\le a<N}\chi(m+p^n a)a.
		\end{align}
		
		When $p^n\equiv 1\bmod N$, the above formula becomes
		\begin{align}
			\label{equation.period-special}
			\mu_\chi(m+p^n\Z_p) = -L(0,\chi) + \sum_{1\le a<m} \chi(a).
		\end{align}
	\end{thm}
	
	\begin{proof}
		By \eqref{equation.period-implicit} and Proposition \ref{prop.h(a,m)}, we have $\mu_\chi(m+p^n\Z_p) = \Omega_\chi(m,n) = \frac{1}{N}\sum_{1\le a<N}\chi(a)(\frac{a-m}{p^n})^\flat_N$. Thus
		\begin{align*}
			\mu_\chi(m+p^n\Z_p) = \frac{1}{N}\sum_{1\le a< N} \chi(a+m)(\frac{a}{p^n})^\flat_N = \frac{1}{N}\sum_{1\le a< N} \chi(m+ a p^n) a.
		\end{align*}
		
		Now assume $p^n\equiv 1\bmod N$. To start, note that since $h(a,m)$ is periodic in $m$ with period $N$, so is $\Omega_\chi(m,n)$ by \eqref{equation.period-implicit}, as long as $0\le m<p^n$. As such, since the same periodicity in $m$ is also enjoyed by the right-hand side of \eqref{equation.period-special}, we only need to prove \eqref{equation.period-special} when $0\le m<N$. In this case:
		\begin{align*}
			\Omega_\chi(m,n)
			= \frac{1}{N}\sum_{1\le a<N}\chi(a)(a-m)^\flat_N
			= \frac{1}{N}\sum_{1\le a<N}\chi(a)(a-m) + \frac{1}{N}\sum_{1\le a<m}\chi(a)N = -L(0,\chi) + \sum_{1\le a<m}\chi(a).
		\end{align*}
	\end{proof}
	
	\begin{cor}
		Let $\chi$ be as in Theorem \ref{thm.explicit-period-formula}, and $\psi$ be a Dirichlet character of conductor a power of $p$. Let $q>1$ be the smallest power such that $q\equiv 1\bmod N$. Then
		\begin{align}
			\label{equation.sumexpr}
			\begin{split}
				-L_p(-s,\chi\psi\omega)
				&=\lim_{n\to \infty} \sum_{1\le m< q^n,p\nmid m} \left(\sum_{1\le a< m^\flat_N}\chi(a)\right) \psi(m)\chx{m}^s\\
				&=\lim_{n\to \infty} \sum_{1\le a<N}\chi(a)\sum_{\substack{1\le m<q^n\\ p\nmid m,m^\flat_N>a}}\psi(m)\chx{m}^s.
			\end{split}
		\end{align}
	\end{cor}
	
	\begin{proof}
		This follows from taking the Riemann sums of \eqref{equation.integral-representation-dirichlet} with respect to the cosets of $q^n\Z_p$ for all $n\ge 0$, and the two simple facts below; the second is standard in $p$-adic analysis (see \cite[§34]{Schi}).
		\begin{enumerate}
			\item $\sum_{1\le a<m}\chi(a) = \sum_{1\le a<m^\flat_N} \chi(a)$ since $\chi$ is nontrivial;
			
			\item $\lim_{n\to \infty} \sum_{1\le m< q^n, p\nmid m} \psi(m)\chx{m}^s = 0$, since the function $\psi(m)\chx{m}^s$ is continuous on $\Z_p^\times$, and thus on $\Z_p$ by extending by zero.
		\end{enumerate}
	\end{proof}
	
	\begin{rem}
		We may change $\flat$ to $\sharp$ and still get an equality, since $\sum_{1\le a< N} \chi(a) =0$.
	\end{rem}
	
	\begin{rem}
		The limit sign of the second equality of \eqref{equation.sumexpr} can be taken inside. Namely,
		\[\lim_{n\to \infty} \sum_{\substack{1\le m<q^n\\ p\nmid m,m^\flat_N>a}} \psi(m)\chx{m}^s\]
		
		exists. This is discussed in Appendix \ref{appx.convergence}.
	\end{rem}
	
	\begin{rem}
		With some algebraic manipulation, one can show that the 2-regularized sum expression given by Theorem 2.4 of \cite{De09} follows from the second period formula \eqref{equation.period-special}, by deriving the following 2-regularized version ($p$ is odd and $f'$ is the order of $p$ in $(\Z/2N)^\times$):
		\begin{align}
			\label{equation.delbourgo-dirichlet}
			\Omega_{L_{\chi,2}}(m,f'n) = L(0,\chi) + \sum_{1\le a<m} \chi(a) - 2\sum_{1\le a<m/2} \chi(a),
		\end{align}
		
		where $L_{\chi,2}(t) = L_\chi(t) - 2L_\chi(t^2)$ is the rational function used in \cite{De09}.
	\end{rem}
	
	\begin{rem}
		While we have only focused on the case of $p^n\equiv 1$ in the sum expression \eqref{equation.sumexpr}, there are also sum expressions for other residues in $(\Z/N)^\times$. For instance, when $p \equiv 2\bmod 3$ and $\chi$ is the quadratic character of conductor $3$, we have
		\begin{align*}
			L_p(-s,\chi\psi\omega) = \lim_{n\to \infty} \sum_{\substack{1\le m<p^{2n+1}, p\nmid m\\ m\equiv 1\bmod 3}} \psi(m)\chx{m}^s.
		\end{align*}
	\end{rem}
	
	\begin{rem}
		Since $\mu_\chi$ is valued in $\Z_p$ and $s\in \Z_p$, the $n$-th partial sum of \eqref{equation.sumexpr}, being the Riemann sum of division by $q^n\Z_p$, is equal to $-L_p(-s,\chi\psi\omega)$ modulo $q^n$. In practice, the computational complexity of the sum expression can be cut by half by noting that $\chi\psi\omega$ is even, and that $\chx{m}^s\equiv \chx{p^n - m}^s\bmod p^n$.
	\end{rem}
	
	\begin{rem}
		\label{remark.leopoldt}
		Using Leopoldt's formula \cite[p61, Theorem 3]{Iw}, one can deduce the following analogous formulas by letting $s=1$:
		\begin{align*}
			\label{}
			-(1-\frac{1}{p})\log_p N = \lim_{n\to \infty} \sum_{1\le m< q^n, p\nmid m} \frac{m^\sharp_N}{m}
		\end{align*}
		
		\begin{align*}
			\log_p\left(\frac{1-\zeta_{p^h}^b}{1-\zeta_{p^h}^{bN}}\right) = \lim_{n\to \infty} \sum_{1\le m< q^n, p\nmid m} \frac{\zeta_{p^h}^{bm} m^\sharp_N}{m}\quad \text{ for }h\ge 1, b\in (\Z/p^h)^\times
		\end{align*}
		
		\begin{align*}
			\log_p\left(\frac{1-\zeta_N^c}{(1-\zeta_N^{cp})^{1/p}}\right) = \lim_{n\to \infty} \sum_{1\le m< q^n, p\nmid m} \frac{\sum_{1\le a<m^\flat_N} \zeta_N^{ca}}{m}\quad \text{ for } c\in (\Z/N)^\times
		\end{align*}
		
		\begin{align*}
			\log_p(1-\zeta_{p^h}^b\zeta_N^c) = \lim_{n\to \infty} \sum_{1\le m< q^n, p\nmid m} \frac{\zeta_{p^h}^{bm}\sum_{1\le a<m^\flat_N}\zeta_N^{ca}}{m}\quad \text{ for }h\ge 1, b\in(\Z/p^h)^\times,c\in(\Z/N)^\times.
		\end{align*}
		
		When $p$ is odd and $N=2$, starting from the first formula, we can easily obtain the following variant
		\begin{align*}
			-2(1-\frac{1}{p})\log_p 2 = \lim_{n\to \infty} \sum_{1\le m<\frac{p^n}{2},p\nmid m} \frac{1}{m}.
		\end{align*}
		
		An interesting question is the following: given $n$, what is the largest $d(n)\in \Z_{\ge 0}$ such that
		\begin{align*}
			-2(1-\frac{1}{p})\log_p 2 \equiv \sum_{1\le m<\frac{p^n}{2},p\nmid m} \frac{1}{m} \pmod{p^{d(n)}}?
		\end{align*} 
		
		While it is automatic from the Riemann sum perspective that $d(n)\ge n$, from some small numerical experiments we find that $d(n)\ge 2n-1$. Can one show that there exists a constant $C>0$, such that for all $p>2$ and $n\ge 1$, $d(n)\ge 2n-C$?
	\end{rem}
	
	\begin{rem}
		\label{rem.Iwasawa}
		Finally we remark on an important application of explicit period formulas on Iwasawa invariants. Basically, for any $p$-adic measure $\mu$ attached to some Kubota-Leopoldt $p$-adic $L$-function, by choosing a topological generator $\kappa$ of $1+p\Z_p$, we can form a power series in $R[[t-1]]$ via the integral $\int_{\Z_p^\times} t^{-\frac{\log_p x}{\log_p \kappa}} \mu(x) $, which in turn is the Iwasawa power series of the corresponding $p$-adic $L$-function. If we take the $p^n\Z_p$-Riemann sum of the integral, we then get a well-defined element in $R[t-1]/(t^{p^n}-1)$, whose coefficients encode the information of Iwasawa $\mu$ and $\lambda$ invariants modulo $p^n$. This observation is classically exploited by \cite{Fe78} and \cite{FW79}, where the coefficients are obtained from Stickelberger elements. For the link between Stickelberger elements and sum expressions we refer the reader to Appendix \ref{appx.stickelberger}.
	\end{rem}
	
	\section{Application: Derivative Formulas at $s=0$}
	\label{sec.application}
	
	As promised we now reprove the Ferrero-Greenberg formula using the sum expression for $p$-adic Dirichlet $L$-functions, based on an eponymous trick recalled in Appendix \ref{appx.FG}. Even better, our approach also leads to formulas of higher derivatives expressed in terms of certain elementary functions. For the original and other existing proofs, see \cite{FG78,Ko79,Wa81}
	
	\vspace{3mm}
	
	As is customary, let $\chi$ be a non-trivial Dirichlet character of conductor $N$ that is prime to $p$, and let $\psi$ be one of a $p$-power conductor with $\chi\psi(-1)=-1$ so that $L_p(s,\chi\psi\omega)$ is not identically zero. For simplicity in this section we will drop the subscript from the notation $m^{\flat/\sharp}_N$. Changing $\flat$ to $\sharp$ in \eqref{equation.sumexpr}, we have
	\begin{align}
		\label{equation.taylor-pre}
		\begin{split}
			-L_p(-s,\chi\psi\omega) 
			&= \lim_{n\to \infty}\sum_{1\le a<N}\chi(a)
			\sum_{\substack{1\le m< q^n\\ p\nmid m, m^\sharp > a}}\psi(m)\chx{m}^s\\
			&= \lim_{n\to \infty}\sum_{1\le a<N}\chi(a)
			\sum_{\substack{1\le m<q^n\\ p\nmid m,m^\sharp >a}}\psi(m)
			\sum_{k\ge 0} (\log_{p} m)^k\frac{s^k}{k!}\\
			&
			= \lim_{n\to \infty}\sum_{1\le a<N}\chi(a)\sum_{k\ge 0}\frac{s^k}{k!}\sum_{\substack{1\le m<q^n\\ p\nmid m,m^\sharp >a}}\psi(m)(\log_{p} m)^k.
		\end{split}
	\end{align}
	
	We now show $\lim_{n\to \infty}\sum_{\substack{1\le m<q^n\\ p\nmid m,m^\sharp >a}} \psi(m)(\log_p m)^k$ exists for all $k\ge 0$, so that we may take the limit inside and obtain a Taylor expansion. To achieve this, we invoke the Ferrero-Greenberg permutation (see Appendix \ref{appx.FG}), supposing $q^n$ is larger than the conductor of $\psi$:
	\begin{align*}
		\sum_{\substack{1\le m<q^n\\ p\nmid m,m^\sharp >a}}\psi(m)(\log_{p} m)^k
		&\equiv \sum_{\substack{1\le m<q^n\\ p\nmid m,m^\sharp >a}}
		\psi(\iota(m)N)[\log_p\iota(m) + \log_p N]^k \pmod{q^n}\\
		&=\sum_{\substack{1\le m<\frac{N-a}{N}q^n\\ p\nmid m}} \psi(m)\psi(N)[\log_p m+\log_p N]^k \\
		& = \psi(N)\sum_{0\le i\le k}\binom{k}{i}(\log_{p} N)^{k-i}\sum_{\substack{1\le m<\frac{N-a}{N}q^n\\ p\nmid m}} \psi(m)(\log_{p} m)^i.
	\end{align*}
	
	As such, we conclude that $\lim_{n\to \infty} \sum_{\substack{1\le m<q^n\\ p\nmid m,m^\sharp >a}} \psi(m)(\log_{p} m)^k$ exists from the existence of
	\[
	\lim_{n\to \infty} \sum_{\substack{1\le m<\frac{N-a}{N}q^n\\ p\nmid m}} \psi(m)(\log_p m)^i,
	\]
	
	which in turn is guaranteed by Example \ref{eg.proportioanl-sum}. Indeed, as we have seen in that example, the above limit is exactly $\Lambda_{\psi\log_p^i}(\frac{a}{N})$, whereby
	\begin{align}
		\label{equation.coefficient}
		\begin{split}
			\lim_{n\to \infty}\sum_{\substack{1\le m<q^n\\ p\nmid m,m^\sharp >a}} \psi(m)(\log_{p} m)^k = \psi(N)\sum_{0\le i\le k}\binom{k}{i}(\log_{p} N)^{k-i} \Lambda_{\psi\log_p^i}(\frac{a}{N}).
		\end{split}
	\end{align}
	
	Combining \eqref{equation.taylor-pre} and \eqref{equation.coefficient} we then have:
	
	\begin{thm}
		Let $\chi$ be a Dirichlet character of conductor $N>1$ that is prime to $p$, and $\psi$ be one of a $p$-power conductor. Then we have
		\begin{align}
			\label{equation.taylor}
			-L_p(-s,\chi\psi\omega) = \psi(N) \sum_{k\ge 0}\frac{s^k}{k!}
			\sum_{1\le a<N}\chi(a) \left[\sum_{0\le i\le k}\binom{k}{i}(\log_p N)^{k-i}\Lambda_{\psi\log_p^i}(\frac{a}{N})\right].
		\end{align}
	\end{thm}
	
	Next we compare our computation with the Ferrero-Greenberg formula: Let $\psi=\1$ be the trivial character and $\chi$ be odd, then for $k=1$:
	\begin{align*}
		\begin{split}
			\sum_{0\le i\le 1} \binom{1}{i}(\log_p N)^{1-i}\Lambda_{\log_p^i}(\frac{a}{N})
			&= \Lambda_{1}(\frac{a}{N})\log_{p} N + \Lambda_{\log_p}(\frac{a}{N})\\
			&= \left[\frac{a}{N} - 1 -  V(\frac{a}{N}-1)\right]\log_{p} N + \log_p\Gamma_p(\frac{a}{N}),
		\end{split}
	\end{align*}
	
	where the last equality follows from our discussions in Example \ref{eg.verschiebung} and \ref{eg.Gamma}. By \eqref{equation.taylor} we thus have
	\begin{align}
		\label{equation.pre-derivative}
		\begin{split}
			L'_p(0,\chi\omega) &= \sum_{1\le a<N}\chi(a)\left[\log_{p} N(\frac{a}{N} - 1 - V(\frac{a}{N}-1)) + \log_{p}\Gamma_p(\frac{a}{N})\right]\\
			&= \sum_{1\le a<N}\chi(a)\log_{p}\Gamma_p(\frac{a}{N}) + \log_{p} N\sum_{1\le a<N}\chi(a)\frac{a}{N} - \log_{p} N\sum_{1\le a<N}\chi(a)V(\frac{a}{N}-1)\\
			&= \sum_{1\le a< N}\chi(a)\log_p\Gamma_p(\frac{a}{N}) - L(0,\chi)\log_p N - \log_p N\sum_{1\le a< N}\chi(a)V(\frac{a}{N}-1).
		\end{split}
	\end{align}
	
	Combining \eqref{equation.pre-derivative} with the lemma below, we have recovered the following formula of Ferrero-Greenberg \cite[Proposition 1]{FG78}:
	\begin{align}
		L'_p(0,\chi\omega) = \sum_{1\le a< N}\chi(a)\log_p\Gamma_p(\frac{a}{N}) -(1-\chi(p))L(0,\chi)\log_p N.
	\end{align}
	
	\begin{lem}
		If $\chi$ is odd, then
		\begin{align}
			\label{equation.verschiebung}
			\sum_{1\le a<N}\chi(a) V(\frac{a}{N}-1) = \chi(p) \sum_{1\le a<N}\chi(a)\frac{a}{N}.
		\end{align}
	\end{lem}
	
	\begin{proof}
		First we see that $V(\frac{a}{N}-1) =\frac{1}{p}\left[\frac{a}{N}-1 -\left(\frac{a}{N}-1\right)^\flat_p\right]$. Write $(\frac{a}{N}-1)^\flat_p = \frac{-(N-a)+p h(0,N-a)}{N}$ as in §\ref{sec.sumexpr}. Then Proposition \ref{prop.h(a,m)} shows that $h(0,N-a) = (-a/p)^\flat_N$. Thus
		\begin{align*}
			\sum_{1\le a<N}\chi(a) V(\frac{a}{N}-1)
			&= -\sum_{a\in (\Z/N)^\times} \chi(a) \frac{(-a/p)^\flat_N}{N}\\
			&= -\sum_{1\le a<N} \chi(-ap) \frac{a}{N}\\
			&= \chi(p) \sum_{1\le a<N}\chi(a)\frac{a}{N}.
		\end{align*}
	\end{proof}

	\appendix

	\section{Ferrero-Greenberg permutation}
	\label{appx.FG}
	
	Notation as set in §\ref{subsec.introduction-notation}. Let $M_n$ be the set $\{1\le m< q^n: p\nmid m\}$, which has two filtrations:
	\begin{itemize}
		\item $M_n = \Phi^n_0 \supsetneq \Phi^n_1\supsetneq\cdots \Phi^n_{N-1} \supsetneq \Phi^n_N = \emptyset$, where for $0\le a\le N$, $\Phi^n_a = \{1\le m< q^n: p\nmid m, m^\sharp_N > a\}$.
		
		\item $M_n = \Psi^n_0 \supsetneq \Psi^n_1 \supsetneq\cdots \Psi^n_{N-1} \supsetneq \Psi^n_N = \emptyset$, where for $0\le a\le N$, $\Psi^n_a = \{1\le m < \frac{N-a}{N}q^n: p\nmid m\}$.
	\end{itemize}
	
	Following \cite{FG78} we define the map $\iota: M_n \to M_n$,
	\[\text{for }m = m^\sharp_N + h N,\ \iota(m) = h+1 + (N-m^\sharp_N)\frac{q^n-1}{N}.\]
	
	We record here the standard properties of $\iota$.
	\begin{prop}
		The Ferrero-Greenberg map $\iota$ is a well-defined permutation on the finite set $M_n$. More precisely, we have:
		\begin{enumerate}
			\item[(i)] If $m\in M_n$, then $\frac{q^n-1}{N} \cdot m \equiv -\iota(m) \bmod q^n$, i.e., $m\equiv N\iota(m) \bmod q^n$.
			
			\item[(ii)] For all $0\le a\le N$, $\iota$ induces a bijection between $\Phi^n_a$ and $\Psi^n_a$. 
		\end{enumerate}
	\end{prop}
	
	\begin{proof}
		See the proof of \cite[Lemma 1]{FG78}.
	\end{proof}

	Here is an illutration of how $\iota$ looks like for $p=5$ and $N=3$:
	We may arrange $M_2=\{1\le m<25 : 5\nmid m\}$ in the table
	\[
	\begin{array}{|*{8}{c|}}
		\hline
		1 & 4 & 7 &  & 13 & 16 & 19 & 22\\
		\hline
		2 && 8& 11& 14& 17& & 23\\
		\hline
		3& 6& 9& 12& & 18& 21& 24\\ \hline
	\end{array}
	\]
	
	which under $\iota$ is mapped to
	\[
	\begin{array}{|*{8}{c|}}
		\hline
		17 & 18 & 19& & 21 & 22 & 23 & 24\\
		\hline
		9 & & 11& 12& 13& 14& & 16\\
		\hline
		1& 2& 3& 4& & 6& 7& 8\\ \hline
	\end{array}
	\]

	\section{Convergence of Sum Expressions}
	\label{appx.convergence}
	
	It may appear intriguing that the limits that show up in sum expressions exist. Our goal here is to give an elementary proof to their convergence. More specifically, in both sum expressions stated in Theorem \ref{thm.main}, by grouping terms according to their residues modulo $N$, it suffices to prove for any $0\le a\le N$ the existence of the following
	\begin{align*}
		\lim_{n\to \infty} \sum_{\substack{1\le m<q^n\\ p\nmid m,m^\sharp >a}} \psi(m)\chx{m}^s.
	\end{align*}
	
	By using Ferrero-Greenberg permutation in the style of §\ref{sec.application}, it reduces to prove the existence of 
	\begin{align}
		\label{equation.limit-proportional}
		\lim_{n\to \infty} \sum_{\substack{1\le m<\frac{N-a}{N}q^n\\ p\nmid m}} \psi(m)\chx{m}^s.
	\end{align}
	
	To show this, we invoke some standard fact of $p$-adic analysis \cite{Schi}. Let $f$ be a continuous function on $\Z_p$ valued in $\C_p$. Then we may uniquely solve the difference equation (\textit{ibid.}, Theorem 34.1)
	\begin{align*}
		F(x+1)-F(x) = f(x),
	\end{align*}
	
	so that $F$ is also continuous on $\Z_p$ and $F(0)=0$. Now, if $f$ is only a continuous function on $\Z_p^\times$, we may extend $f$ on $\Z_p$ by zero, and denote the resulting continuous function by $f_!$. As such we let $\Lambda_f$ be the unique solution to the difference equation of $f_!$. In this way, the existence of \eqref{equation.limit-proportional} follows from Example \ref{eg.proportioanl-sum} below.
	
	\begin{eg}
		\label{eg.verschiebung}
		Let $f(x) = 1$ be the constant function. Then for all $n\in \Z_{>0}$, $\Lambda_f(n) = \sum_{1\le m< n, p\nmid m} 1 = n-1 - \lfloor\frac{n-1}{p}\rfloor$. The function $n\mapsto \lfloor\frac{n}{p}\rfloor$ sends $a_0 + a_1 p +\cdots a_r p^r$ to $a_1 + a_2 p + \cdots + a_r p^{r-1}$, and therefore is extended to the familiar (continuous) Verschiebung operator  $V$ on the Witt ring $W(\F_p)=\Z_p$. Therefore, $\Lambda_f(x) = x-1 - V(x-1)$.
	\end{eg}
	
	\begin{eg}
		\label{eg.Gamma}
		Let $f(x) = \log_p x$. Then we have $\Lambda_f(x) = \lim_{n>0,n\to x} \log_p(\prod_{1\le m\le n,p\nmid m} m) = \log_{p}\Gamma_p(x+1)$, where $\Gamma_p(x)$ is the Morita $p$-adic Gamma function \cite[§14.1]{Lang}. 
	\end{eg}
	
	\begin{eg}
		\label{eg.proportioanl-sum}
		Finally let $f$ be continuous on $\Z_p^\times$, $0\le\frac{a}{N}\le 1$ with $\gcd(N,p) =1$, and $q = p^f\equiv 1\bmod N$. Then
		\begin{align*}
			\lim_{n\to \infty} \sum_{1\le m<\frac{N-a}{N}q^n , p\nmid m} f(m) = \lim_{n\to \infty} \Lambda_f\left(\frac{(N-a)q^n+a^\flat_N}{N}\right) = \Lambda_f(\frac{a^\flat_N}{N}).
		\end{align*}
	\end{eg}
	
	\section{An Elementary Method to Sum Expressions}
	\label{appx.elementary}
	
	In this appendix we sketch how sum expressions can be proved by using only some elementary $p$-adic analysis, following the discussions in previous appendices. We will only treat $p$-adic zeta functions with odd $p$, since they are technically easier.
	
	\subsection{Preliminaries on power sums.}
	
	For $k,n\in \Z_{\ge 0}$ we let $S_k(n) = \sum_{0\le d\le n} d^k$. For example, we have $S_0(n) = n+1$, $S_1(n) = \frac{n(n+1)}{2}$ and $S_2(n) = \frac{n(n+1)(2n+1)}{6}$. In general, if we write
	\begin{align*}
		\frac{Xe^{yX}}{e^X-1} = \sum_{r\ge 0}\frac{X^r}{r!}B_r(y),
	\end{align*}
	
	then
	\begin{align*}
		S_r(n) = \frac{B_{r+1}(n+1) - B_{r+1}}{r+1}.
	\end{align*}
	
	Here $B_r(0) = B_r$ is the usual Bernoulli number.
	
	\vspace{3mm}
	
	Following Appendix \ref{appx.convergence}, for $k\in \Z_{\ge 0}$ we consider the function $\Lambda_{k}(x) = \lim_{n\in\Z_{>0},n\to x}\sum_{0\le d< n, p\nmid d} d^k$. For any $n\in \Z_{>0}$ we have
	\begin{align*}
		\Lambda_{k}(n) = S_k(n-1) - p^k S_k(\lfloor \frac{n-1}{p} \rfloor) = \frac{B_{k+1}(n) - p^k B_{k+1}(\lfloor \frac{n-1}{p}\rfloor +1)}{k+1} - (1-p^k)\frac{B_{k+1}}{k+1}.
	\end{align*}
	
	Hence
	\begin{align}
		\label{equation.Lambda-bernoulli}
		\Lambda_k(x) = \frac{B_{k+1}(x) - p^k B_{k+1}(V(x-1)+1)}{k+1} - (1-p^k)\frac{B_{k+1}}{k+1},
	\end{align}
	
	where $V$ is the Verschiebung operator.

	\subsection{The proof.}
	
	Let $N$ and $q$ be as in §\ref{subsec.introduction-notation}. For an even number $k$, define:
	\begin{align*}
		L_k(s) 
		= \lim_{n\to \infty}\sum_{1\le m< q^n, p\nmid m} \omega^{k-1}(m)\chx{m}^{-s}m^\sharp_N.
	\end{align*}
	
	Then, by density, to prove the sum expression for $-(1-\omega^k(N)\chx{N}^{-s+1})L_p(s,\omega^k)$, it suffices to show that for all $r\in \Z_{>0}$ such that $r \equiv k\bmod p-1$, $L_k(1-r)=-(1-N^r)L_p(1-r,\omega^k)$. First note that by the Ferrero-Greenberg permutation (applied to $\Phi^n_{a-1}\setminus\Phi^n_a$), we have
	\begin{align}
		\label{equation.L_k-special}
		L_k(1-r) = N^{r-1}\lim_{n\to \infty}\sum_{1\le a\le N} a
		\sum_{\frac{N-a}{N}q^n\le m< \frac{N-a+1}{N}q^n, p\nmid m}
		m^{r-1}.
	\end{align}
	
	Using \eqref{equation.Lambda-bernoulli}, for all $a$ in $[1,N]$, it is not hard to see that
	\begin{align*}
		&\sum_{\frac{N-a}{N}q^n\le m< \frac{N-a+1}{N}q^n, p\nmid m} m^{r-1}\\
		=& \frac{1}{r}\left\{B_r\left(\frac{a-1}{N}\right) - B_r\left(\frac{a^\flat_N}{N}\right) - p^{r-1}\left[B_r\left(V\left(-\frac{N-a+1}{N}\right)+1\right) - B_r\left(V\left(-\frac{N-a^\flat}{N}\right)+1\right)\right]\right\}.
	\end{align*}
	
	To simplify, note the following Abel-type summations:
	\begin{align*}
		\sum_{1\le a\le N} a \left[B_r\left(\frac{a-1}{N}\right) - B_r\left(\frac{a^\flat_N}{N}\right)\right] = \sum_{0\le a<N} B_r\left(\frac{a}{N}\right) - N B_r,
	\end{align*}
	
	and
	\begin{align*}
		\sum_{1\le a\le N} a
		\Big[ B_r(V(-\frac{N-a+1}{N})+1) - B_r(V(-\frac{N-a^\flat_N}{N})+1)\Big]
		=&\sum_{1\le a\le N} B_r(V(-\frac{N-a^\flat}{N})+1) - N B_r(V(-1)+1)\\
		=&\sum_{1\le a\le N} B_r(\frac{N-(a/p)^\sharp_N}{N}) - N B_r\\
		=&\sum_{0\le a<N}B_r(\frac{a}{N}) - N B_r.
	\end{align*}

	Combinining these equalities, we have
	\begin{align*}
		L_k(1-r) = N^{r-1}(1-p^{r-1})\frac{1}{r}\left[\sum_{0\le a<N}B_r(\frac{a}{N}) - N B_r\right].
	\end{align*}
	
	To proceed we use the identity 
	\begin{align*}
		\sum_{0\le a<N} B_r(\frac{a}{N}) = N^{1-r} B_r,
	\end{align*}
	
	which can be seen from the Taylor expansion of
	\begin{align*}
		\sum_{0\le a<N} \frac{X e^{\frac{a}{N} X}}{e^X-1} = \frac{X}{e^{X/N}-1}.
	\end{align*}
	
	In turn,
	\begin{align*}
		L_k(1-r) = N^{r-1}(1-p^{r-1})\frac{1}{r}(N^{1-r}-N)B_r = (1-N^r)(1-p^{r-1})\frac{B_r}{r}.
	\end{align*}
	
	This concludes $L_k(1-r)= -(1-N^r)L_p(1-r,\omega^r)$ by \eqref{equation.interpolation-classical}.

	\section{From Stickelberger Elements to Sum Expressions}
	\label{appx.stickelberger}
	
	In this appendix, we will show that the sum expressions for $p$-adic $L$-functions can be directly derived from Stickelberger elements. Our main reference is \cite{Iw69}, complemented by Chapter 6 of \cite{Iw}. For simplicity we assume $p\ge 3$, for $p=2$ the argument runs \textit{mutatis mutandis}.
	
	\subsection{Background on Stickelberger elements.} 
	
	Let $\chi$ be a character of conductor $N\ge 1$ such that $\gcd(N,p)=1$, and $\psi$ be one of conductor $p^e$ for some $e\ge 0$. For any $n\ge e-1$, the $n$-th Stickelberger element attached to $(\chi,\psi)$ is defined to be \cite[p200]{Iw69}
	\begin{align*}
		\xi_n = -\frac{1}{Np^{n+1}}\sum_{1\le a<Np^{n+1}, p\nmid a} \chi\psi(a)a [\gamma(a)^{-1}] \in \bar{\Q}_p[\Gamma_n].
	\end{align*}
	
	Here $\Gamma_n = 1+p\Zp/1+p^{n+1}\Zp$, $\gamma: (\Z/N)^\times\times \Z_p^\times \to 1+p\Z_p$ is the natual projection, and $[\gamma(a)]$ is the element in $\bar{\Q}_p[\Gamma_n]$ corresponding to $\gamma(a)$. When $\chi$ is non-trivial, $\xi_n$ is known to be in $\bar{\Z}_p[\Gamma_n]$ \cite[pp.~75-76]{Iw}. When $\chi$ is trivial, hence $N=1$, a regularization process is required: Let $c\in \Z_p^\times$ and define the regularized Stickelberger element 
	\begin{align*}
		\eta_{c,n} = (1-\psi(c)c[\gamma(c)^{-1}])\xi_n.
	\end{align*}
	
	It can be shown that $\eta_{c,n}$ has coefficients in $\bar{\Z}_p$; explicitly
	\begin{align*}
		\eta_{c,n} = -\frac{1}{p^{n+1}}\sum_{1\le a<p^{n+1}, p\nmid a} \left\{\psi(a) a[\gamma(a)^{-1}] - \psi(ac)ac[\gamma(ac)^{-1}] \right\}
		= \sum_{1\le a<p^{n+1}, p\nmid a}\psi(ac)\frac{ac-(ac)^\flat_{p^{n+1}}}{p^{n+1}} [\gamma(ac)^{-1}].
	\end{align*}
	
	As discussed in \cite[pp.~72-73]{Iw}, when $\chi\psi$ is odd, the Stickelberger elements are compatible with respect to projections, and thus we may consider the limits $\xi=\lim_{n\to\infty} \xi_n$ (when $\chi$ is non-trivial) and $\eta_c=\lim_{n\to\infty}\eta_{c,n}$ (when $\chi$ is trivial). Now, let $\varphi_s: \plim_n \bar{\Z}_p[\Gamma_n] \to \C_p$ be the specialization map sending $\gamma(a)$ to $\gamma(a)^{-s}$, then \cite[§3]{Iw69} has shown that
	\begin{align}
		\label{equation.interpolation-stickelberger-zeta}
		\varphi_s(\eta_c) = (1-\psi\omega(c)\chx{c}^{1+s})L_p(-s,\psi\omega),
	\end{align}
	
	and 
	\begin{align}
		\label{equation.interpolation-stickelberger-dirichlet}
		\varphi_s(\xi) = L_p(-s,\chi\psi\omega).
	\end{align}
	
	\subsection{Re-deriving sum expressions: zeta case.}
	
	By approximating $\eta_c$ by $\eta_{c,n}$ for all $n\ge e$, we have the following limit by \eqref{equation.interpolation-stickelberger-zeta}:
	\begin{align*}
		(1-\psi\omega(c)\chx{c}^{1+s})L_p(-s,\psi\omega)
		&= \lim_{n\to\infty} \sum_{1\le a<p^n, p\nmid a}
		\left(\frac{ac-(ac)^\flat_{p^n}}{p^n}\right)\psi(ac)\chx{ac}^s\\
		&= \lim_{n\to\infty} \sum_{1\le a<p^n, p\nmid a}
		-\left(\frac{(a/c)^\flat_{p^n}c-a}{p^n}\right)\psi(a)\chx{a}^s - \frac{c-1}{2}\lim_{n\to\infty}\sum_{1\le a<p^n, p\nmid a}\psi(a)\chx{a}^s\\
		&= -\lim_{n\to\infty} \sum_{1\le a<p^n, p\nmid a}\left(\frac{a}{p^n}-\frac{1}{2} - c \left[\frac{(a/c)^\flat_{p^n}}{p^n}-\frac{1}{2}\right]\right) \psi(a)\chx{a}^s.
	\end{align*}
	
	As such, to obtain the sum expression when $c\in \Z_{>1}$, one may proceed by employing the identity $(\frac{a}{c})^\flat_{p^n} = \frac{a+p^n(-a/p^n)^\flat_{c}}{c}$.
	
	\subsection{Re-deriving sum expressions: Dirichlet case.}
	
	Now assume the character $\chi$ is non-trivial, so we may take the limit over $\xi_n$'s as they are already bounded. As such, by \eqref{equation.interpolation-stickelberger-dirichlet}:
	\begin{align*}
		L_p(-s,\chi\psi\omega)
		=& -\lim_{n\to\infty} \frac{1}{Np^n} \sum_{1\le a<Np^n, p\nmid a} \chi\psi(a) a\chx{a}^s\\
		=& -\lim_{n\to\infty} \frac{1}{Np^n} \sum_{1\le a<p^n, p\nmid a} \psi(a)\chx{a}^s
		\sum_{0\le d<N} 
		\chi(dp^n+a) (dp^n+a)\\
		=& -\lim_{n\to\infty} \sum_{1\le a<p^n, p\nmid a} \psi(a)\chx{a}^s\left[\frac{1}{N}\sum_{0\le d<N} \chi(a+dp^n)d \right]\\
		=& -\lim_{n\to\infty}\sum_{1\le a<p^n, p\nmid a} \psi(a)\chx{a}^s \mu_\chi(a+p^n\Zp).
	\end{align*}
	
	\begin{rem}
		\label{rem.stickelberger}
		It is no coincidence that, in the calculations above, the coefficients of Stickelberger elements give rise to values of the periods attached to $\mu_\chi$. Indeed, define the unconventional Stickelberger elements by (\textit{cf.}~\cite[§2.1]{Lang})
		\begin{align*}
			\Theta_n = \sum_{a\in \Z/Np^{n+1}} \left(\frac{a^\flat_{Np^{n+1}}}{Np^{n+1}}-\frac{1}{2}\right)\chi(a)[a] \in \bar{\Q}[\Z/Np^{n+1}].
		\end{align*}
		
		Then we can project $\Theta_n$ to $\theta_n \in \bar{\Q}[\Z/p^{n+1}]$. Explicitly,
		\begin{align*}
			\theta_n = \sum_{0\le a<p^{n+1}} [a]\sum_{0\le d<N} \frac{\chi(a+dp^{n+1}) d}{N},
		\end{align*}
		
		which can be thus regarded as an element in $\Z_p[\chi][\Z/p^n]$. Note that the restriction of $\theta_n$ to $\Z_p[\chi][(\Z/p^{n+1})^\times]$ followed by a projection to $\Z_p[\chi][\Gamma_n]$ recovers $-\xi_n$. On the other hand, using the additivity of the first Bernoulli polynomial \cite[p36, \textbf{B 4}]{Lang}, it can be easily shown that the Stickelberger elements $\Theta_n$'s give rise to a distribution on $\plim_n \Z/Np^n = \Z/N \times \Z_p$, and so do $\theta_n$'s on $\Z_p$; in the latter case a measure by integrality. Moreover, this measure is exactly $\mu_{\chi}$, and we get the explicit period formula for free by reading off the coefficients of $\theta_n$'s.
	\end{rem}

	\bibliographystyle{alpha}
	\bibliography{sumExprKL}
\end{document}